\documentclass[11pt]{amsart}

\usepackage{amssymb}

\usepackage{graphicx}

\newtheorem{theorem}{Theorem}
\newtheorem{lemma}[theorem]{Lemma}

\newtheorem{proposition}[theorem]{Proposition}

\theoremstyle{definition}

\theoremstyle{remark}

\newcommand{\A}{{\mathcal A}}

\newcommand{\C}{\mathbb{C}}
\renewcommand{\H}{\mathcal{H}}

\renewcommand{\span}{\mathrm{span}}

\newcommand{\dx}{\partial / \partial x}

\def\<{\left<}
\def\>{\right>}

\def\<{\left<}
\def\>{\right>}
\def\ll{\langle\kern-3pt\langle}
\def\rr{\rangle\kern-3pt\rangle}

\begin{document}


\title[Two counterexamples for power ideals of arrangements]{Two counterexamples for power ideals of hyperplane arrangements.}


\author{Federico Ardila}
\address{Department of Mathematics, San Francisco State University, 1600 Holloway Ave, San Francisco, CA 94110, USA. }
\curraddr{}
\email{federico@sfsu.edu}
\thanks{Supported in part by NSF Award DMS-0801075 and CAREER Award DMS-0956178.}

\author{Alexander Postnikov}
\address{Department of Mathematics, Massachusetts Institute of Technology, 77 Massachusetts Ave, Cambridge, MA 02139, USA.}
\curraddr{}
\email{apost@math.mit.edu}
\thanks{Supported in part by NSF CAREER Award DMS-0504629.}


\date{}

\dedicatory{}

\begin{abstract}
We disprove Holtz and Ron's conjecture that the power ideal $C_{\A,-2}$ of a hyperplane arrangement $\A$ (also called the  internal zonotopal space) is generated by $\A$-monomials. We also show that, in contrast with the case $k \geq -2$, the Hilbert series of $C_{\A,k}$ is not determined by the matroid of $\A$ for $k \leq -6$.
\end{abstract}

\maketitle

\noindent \textbf{Remark.} This note is a corrigendum to our article \cite{AP}, and we follow the notation of that paper.

\bigskip

\section{Introduction.}

Let $\A = \{H_1, \ldots, H_n\}$ be a hyperplane arrangement in a vector space $V$; say $H_i = \{x \mid l_i(x)=0\}$ for some linear functions $l_i \in V^*$. Call a product of (possibly repeated) $l_i$s an \emph{$\A$-monomial} in the symmetric algebra $\mathbb{C}[V^*]$.
Let $\mathrm{Lines}(\A)$ be the set of lines of intersection of the hyperplanes in $\A$.
For each $h \in V$ with $h \neq 0$, let $\rho_{\A}(h)$ be the number of 
hyperplanes in $\A$ 
not containing $h$. Let $\rho = \rho(\A) = \textrm{min}_{h \in V}(\rho_\A(h))$. For all integers $k \geq - (\rho+1)$, consider the \emph{power ideals}:
\[
I_{\A,k} := \left<h^{\rho_\A(h)+k+1} \mid h \in V, h \ne 0\right>, \quad
I'_{\A,k} := \left<h^{\rho_\A(h)+k+1} \mid h \in \mathrm{Lines}(\A)\right>
\]
in the symmetric algebra $\C[V]$. It is convenient to regard the polynomials in $I_{\A,k}$ as differential operators, and to consider the space of solutions to the resulting system of differential equations:
\[
C_{\A,k} = I_{\A,k}^\perp := \left\{f(x) \in \C[V^*] \mid h\left(\frac{\partial}{\partial x}\right)^{\rho_\A(h)+k+1} f(x) = 0 \textrm{ for all } h \ne 0 \right\}
\]
which is known as the \emph{inverse system} of $I_{\A,k}$. Define $C'_{\A,k}$ similarly.
These objects arise naturally in numerical analysis, algebra, geometry, and combinatorics. 
For references, see \cite{AP, HR}. 

One important question is to compute the Hilbert series of these spaces of polynomials, graded by degree, 
as a function of combinatorial invariants of $\A$. Frequently, the answer is expressed in terms of the Tutte polynomial of $\A$. This has been done successfully in many cases. One strategy used 
independently by different authors has been to prove the following: 

\begin{enumerate}
\item[(i)]  There is a spanning set of $\A$-monomials for $C_{\A,k}$.
\item[(ii)] There is an exact sequence $0 \rightarrow C_{\A \backslash H, k}(-1)
\rightarrow C_{\A,k}
\rightarrow C_{\A / H,k}
\rightarrow 0$
of graded vector spaces.
\item[(iii)] Therefore, the Hilbert series  of $C_{\A,k}$ is an evaluation of the Tutte polynomial of $\A$.
\end{enumerate}
Here $\A \backslash H$ and $\A / H$ are the deletion and contraction of $H$, respectively.

For $k \geq -1$, this method works very nicely. Dahmen and Michelli \cite{DM} were the first ones to do this  for $C'_{\A,-1}$. Postnikov-Shapiro-Shapiro \cite{PSS} did it for $C_{\A,0}$, while Holtz and Ron \cite{HR} did it for $C'_{\A,0}$. In \cite{AP} we did it for $C_{\A,k}$ for all $k \geq -1$, and showed that $C'_{\A,0} = C_{\A,0}$ and $C'_{\A,-1} = C_{\A,-1}$.

For $k \leq -3$ this approach does not work in full generality. In \cite{AP} we showed that (i) is false in general for $C_{\A,k}$, and left (ii) and (iii) open, suggesting the problem of measuring $C_{\A,k}$.
For $k \leq -6$, (ii) and (iii) are false, as we will show in Propositions \ref{prop2} and \ref{prop3}, respectively. In fact, we will see that the Hilbert series of $C_{\A,k}$ is not even determined by the matroid of $\A$.

The intermediate cases are interesting and subtle, and deserve further study; notably the case $k=-2$, which Holtz and Ron call the \emph{internal zonotopal space}. In  \cite{HR} they proved (ii) and (iii) and conjectured (i) for $C'_{\A,-2}$. In \cite[Proposition 4.5.3]{AP} -- a restatement of Holtz and Ron's Conjecture 6.1 in \cite{HR} -- we put forward an incorrect proof of this conjecture; the last sentence of our argument is false. In fact their conjecture is false, as we will see in Proposition \ref{prop}.

\section{The case $k=-2$: internal zonotopal spaces.}

Before showing why Holtz and Ron's conjecture is false, let us point out that the remaining statements about $C_{\A, -2}$ that we made in \cite{AP} are true. The easiest way to derive them is to prove that $C_{\A,-2} = C'_{\A,-2}$, and simply note that Holtz and Ron already proved those statements for $C'_{\A,-2}$:

\begin{lemma}
We have $C_{\A, k} = C'_{\A, k}$ for any $k$ with $- (\rho+1) \leq k \leq 0$.
\end{lemma}

\begin{proof}
By \cite[Theorem 4.17]{AP} we have $I_{\A, 0} = I'_{\A, 0}$, so it suffices to show that  $I_{\A, j} = I'_{\A, j}$ implies  that $I_{\A, j-1} = I'_{\A, j-1}$ as long as these ideals are defined. If $I_{\A, j} = I'_{\A, j}$, then for any $h \in V \backslash \{0\}$ we have $h^{\rho_{\A}(h)+j+1} = \sum f_i h_i^{\rho_{\A}(h_i)+j+1}$ for some polynomials $f_i$, where the $h_i$s are the lines of the arrangement. As long as the exponents are positive, taking partial derivatives in the direction of $h$ gives  $h^{\rho_{\A}(h)+j} = \sum g_i h_i^{\rho_{\A}(h_i)+j}$ for some polynomials $g_i$.
\end{proof}

%
%
%
The following result shows that (i) does not hold for $C_{\A,-2}$.

\begin{proposition} \label{prop} \cite[Conjecture 6.1]{HR} \textnormal{is false}: The ``internal zonotopal space" $C_{\A,-2}$ is not necessarily spanned by $\A$-monomials.
\end{proposition}

\begin{proof}
Let $\H$ be the hyperplane arrangement in $\mathbb{C}^4$ determined by the linear forms 
$y_1, y_2, y_3, y_1-y_4, y_2-y_4, y_3-y_4$. We have
\[
I'_{\H, -2} = \langle x_1^1, x_2^1, x_3^1, (\epsilon_1 x_1 + \epsilon_2 x_2 +\epsilon_3 x_4 +x_4)^2\rangle =  \langle x_1, x_2, x_3, x_4^2 \rangle
\]
as $\epsilon_1, \epsilon_2, \epsilon_3$ range over $\{0,1\}$. The other generators of $I_{\H,-2}$ are of degree at least 3, and are therefore in $I'_{\H,-2}$ already, so 
\[
I_{\H, -2} = \langle x_1, x_2, x_3, x_4^2 \rangle, \qquad C_{\H, -2} = \span(1, y_4).
\]
Therefore $C_{\H, -2}$ is not spanned by $\H$-monomials.
\end{proof}

As Holtz and Ron pointed out, if \cite[Conjecture 6.1]{HR} had been true, it would have implied \cite[Conjecture 1.8]{HR}, an interesting spline-theoretic interpretation of $C_{\A,-2}$ when $\A$ is unimodular. The arrangement above is unimodular, but it does not provide a counterexample to \cite[Conjecture 1.8]{HR}. In fact, Matthias Lenz \cite{Lenz} has recently put forward a proof of this weaker conjecture.

\section{The case $k \leq -6$}

In this section we show that when $k \leq -6$, the Hilbert series of $C_{\A,k}$ is not a function of the Tutte polynomial of $\A$. In fact, it is not even determined by the matroid of $\A$. 
Recall that $\rho=\rho(\A) := \min_{h \in V}(\rho_{\A}(h))$. Say $h \in V$ is \emph{large} if it is on the maximum number of hyperplanes, so $\rho_{\A}(h) = \rho$. 

\begin{lemma}
The degree $1$ component of $C_{\A, -\rho}$ is 
\[
(C_{\A, -\rho})_1 = (\span\{h\in V \, : \, h \textrm{ is large}\})^\perp
\]
in $V^*$.
\end{lemma}
\begin{proof}
An element $f$ of $C_{\A, -\rho}$ needs to satisfy the differential equation $
h\left(\dx\right)^{\rho_\A(h)-\rho+1}f(x) = 0$ for all non-zero $h \in V$. If $f$ is linear, this condition is trivial unless $h$ is large; and in that case it says that $f \perp h$.
\end{proof}

\begin{proposition}\label{prop2}
For $k \leq -6$, the Hilbert series of $C_{\A,k}$ is not determined by the matroid of $\A$.
\end{proposition}

\begin{proof}
First assume $k=-2m$. Let $L_1, L_2, L_3$ be three lines through $0$ in $\mathbb{C}^3$ and consider an arrangement $\A$ of $3m$ (hyper)planes consisting of $m$ generically chosen planes $H_{i1}, \ldots, H_{im}$  passing through $L_i$ for $i=1,2,3$. Then $\rho = 2m$ and the only large lines are $L_1, L_2$, and $L_3$. Therefore $\dim (C_{\A, -2m})_1$ equals $1$ if $L_1, L_2, L_3$ are coplanar, and $0$ otherwise. However, the matroid of $\A$ does not know whether $L_1, L_2, L_3$ are coplanar. 

More precisely,  consider two versions $\A_1$ and $\A_2$ of the above construction; in $\A_1$ the lines $L_1, L_2, L_3$ are coplanar, and in $\A_2$ they are not. Then $\A_1$ and $\A_2$ have the same matroid but $\dim (C_{\A_1, -2m})_1 \neq \dim (C_{\A_2, -2m})_1$.

The case $k=-2m-1$ is similar. It suffices to add a generic plane to the previous arrangements.
\end{proof}

\begin{proposition}\label{prop3}
For $k \leq -6$, the sequence of graded vector spaces 
\[
0 \rightarrow C_{\A \backslash H, k}(-1) \rightarrow C_{\A,k} \rightarrow C_{\A / H, k} \rightarrow 0
\]
of \cite[Proposition 4.4.1]{AP} is not necessarily exact, even if $H$ is neither a loop nor a coloop.
\end{proposition}

\begin{proof}
We will not need to recall the maps that define this sequence; we will simply show an example where right exactness is impossible because $\dim (C_{\A,k})_1 = 0 $ and $\dim (C_{\A / H, k})_1 = 1$. We do this in the case $k=-2m$; the other one is similar. 

Consider the arrangement $\A=\A_2$ of the proof of Proposition \ref{prop2} and the plane $H=H_{11}$. We have $\dim (C_{\A, -2m})_1=0$. 
In the contraction $\A/H$, the planes $H_{12}, \ldots, H_{1m}$ become the same line $L_1$ in $H$, while the other $2m$ planes of $\A$  become generic lines in $H$. Therefore $\rho(\A \backslash H) = 2m$ and $(C_{\A / H, -2m})_1=L_1^\perp$ in $H^*$, which is one-dimensional.
%
\end{proof}

\smallskip

\noindent \textbf{Acknowledgments.} We are very thankful to Matthias Lenz for pointing out the error in \cite{AP}, and to Andrew Berget and Amos Ron for their comments on a preliminary version of this note.

\bibliographystyle{amsplain}

\end{document}